\documentclass[12pt,a4paper,oneside]{article}

\usepackage[left=2.5cm,right=2.5cm,top=2.5 cm,bottom=2.5cm]{geometry}

\usepackage[utf8]{inputenc}
\usepackage[T1]{fontenc}
\usepackage[english]{babel}
\usepackage{tikz}
\usepackage{pgfplots}
\pgfplotsset{compat=newest}
\usepgfplotslibrary{fillbetween}
\usepackage{hyperref}
\hypersetup{
    colorlinks = true,
    linkcolor = blue,
    anchorcolor = blue,
    citecolor = blue,
    filecolor = blue,
    urlcolor = blue
    }
\usetikzlibrary{decorations.markings,arrows}
\usepackage{breqn}
\usepackage{multirow}
\usepackage{amsmath}
\usepackage{longtable}
\usepackage{amsfonts}
\usepackage{amssymb}
\usepackage{graphicx}
\usepackage{amsthm}
\allowdisplaybreaks 
\usepackage{systeme}
\usepackage{float}
\usepackage{enumerate}
\usepackage{mathtools}
\usepackage{comment}
\usepackage{makecell} 
\usepackage{tcolorbox}
\usepackage{collectbox}

\usepackage{etoolbox}
\let\bbordermatrix\bordermatrix
\patchcmd{\bbordermatrix}{\left(}{\left[}{}{}
\patchcmd{\bbordermatrix}{\right)}{\right]}{}{}
\newtheorem{theorem}{Theorem} 
 
\newtheorem{cor}{Corollary}
\newtheorem{remark}{Remark}

\numberwithin{equation}{section}

\newcolumntype{C}[1]{>{\centering\arraybackslash$}m{#1}<{$}}

\newcolumntype{R}[1]{>{\raggedleft\arraybackslash$}m{#1}<{$}}

\newcolumntype{L}[1]{>{\raggedright\arraybackslash$}m{#1}<{$}}

\DeclarePairedDelimiterX{\cif}[1]{(}{)}{\delimsize(#1\delimsize)}

\begin{document}

\title{Identities for full-history Horadam sequences}

\author{Tomislav Došlić\\
    Faculty of Civil Engineering, University of Zagreb\\
    Croatia, Zagreb 10000\\
  \text{tomislav.doslic@grad.unizg.hr}\\
 \and
Luka Podrug\\
    Faculty of Civil Engineering, University of Zagreb\\
    Croatia, Zagreb 10000\\
  \text{luka.podrug@grad.unizg.hr}}
\date{}
\maketitle

\begin{abstract}
We prove a master identity for a class of sequences defined by full-history
linear homogeneous recurrences with (non-negative) constant coefficients.
The identity is derived in a combinatorial way, providing thus combinatorial
proofs for many known and new identities obtained as its corollaries. In
particular, we prove several interesting identities for the Pell, the 
Jacobsthal, and the $m$-nacci numbers.
\end{abstract}

\section{Introduction}

The Fibonacci numbers are among the best known and the most researched
combinatorial sequences, as attested by the length and the number of links,
formulas, and references in the corresponding entry (A00045) in the {\em
The On-Line Encyclopedia of Integer Sequences} \cite{OEIS}. It can be argued
that they are the the simplest non-trivial example of a sequence defined by 
a linear recurrence with constant coefficients. Indeed, their defining 
recurrence, $F_{n} = F_{n-1} + F_{n-2}$, with the initial conditions
$F_0 = 0$, $F_1 = 1$, is the simplest interesting example of such recurrences,
with all shorter recurrences with constant coefficients leading to rather
dull geometric sequences. In spite of their definitional simplicity, the
Fibonacci numbers are, seemingly, inexhaustible source of interesting, and
often fascinating, combinatorial results. In particular, they satisfy a vast
number of identities. Literally thousands of them are scattered in the
literature and new ones are discovered and added almost daily.

The Fibonacci numbers, being both simple and interesting, have spawned a
large number of generalizations. Each of them, if at all meaningful, gives
rise to new classes of identities. One type of generalizations consists of 
varying the coefficients (and the initial conditions) of the defining
recurrence while keeping its length at two, leading to the study of the 
so-called Horadam sequences, first introduced by A. F. Horadam in the
early sixties \cite{Horadam0,Horadam1,Horadam2} and intensely studied ever
since. See, for example, Chapter 5 of ref. \cite{andrica}.
Another way to generality is to lengthen the
defining recurrence, leading to the tribonacci (A000073), the tetranacci
(A000078), and, in general, to the $m$-nacci sequences (A092921). Of course, the 
coefficients in the longer recurrences can also vary \cite{cerda}. In this paper we take 
this second type of generalization to its extreme, by taking into account
all previous values of the sequence and by considering the full-history
linear recurrences with constant coefficients. Our goal is to establish
interesting identities for such sequences.

Let  $ $ $m > 0$ be an integer and let $F^{(m)}_n$ denote the $n$-th $m$-nacci
number, i.e., the $n$-th
element of a sequence that satisfies the recursive relation
\begin{align}
F^{(m)}_n=F^{(m)}_{n-1}+\cdots+ F^{(m)}_{n-m} \label{m-nacci recurence}
\end{align}
with initial values $F^{(m)}_{0}=1$ and $F^{(m)}_{n}=0$ for $n<0$. From these
sequences one can obtain the usually indexed $m$-nacci numbers by shifting
indices by $m-1$, but in this paper we will use the unshifted indices to
simplify our expressions. 

In the first part of the paper, we consider a rectangular strip with
dimensions $1\times n$, i.e. the rectangle of height $1$ and length $n$,
and rectangular tiles with dimension $1\times k$, where $k$ is an integer
between $1$ and $m$. We are interested in tiling the strip using that set
of tiles. We denote the number of all possible such tilings by $g_{n}^{(m)}$.
This number is given as a special case of the next theorem by Benjamin
and Quinn \cite{PTRC}.
\begin{theorem} \label{benq}
Let $a_1, a_2, . . . a_m$ be nonnegative integers, and let $A_n$
be the sequence of numbers defined by the recurrence
$A_n=a_{1}A_{n-1}+\cdots+ a_{m}A_{n-m}$ with initial conditions
$A_0=1$ and $A_n=0$ for $n\leq 0$. Then the $n$-th element of this
sequence equals to the number of all possible colored tilings of a board
of length $n$, where the tiles are of length $k$ for $1\leq k\leq m$
and each tile of a length $k$ admits $a_k$ colors.
\end{theorem}

Hence, for a special case of a previous theorem where $a_1=a_2=\cdots=a_m=1$
we have $g_n^{(m)}=F^{(m)}_n$. For $m=2$ we have the classical Fibonacci
numbers (A000045 in the OEIS \cite{OEIS}), for $m=3$ the tribonacci numbers
(A000073), and in general the $m$-nacci numbers. 

In the next part we take a step further and generalize the sequence
$F^{(m)}_{n}$ by defining a new sequence $S_n$ with recursive relation
\begin{align}
S_n=a_{n-1}S_{n-1}+\cdots+a_{1}S_{1} + a_0 S_0
\label{recursion S_n}\end{align}
with initial values 
\begin{align}
S_{0}=1 \text{ and } S_{k}=0 \text{ for }  k<0,
\label{initialvalues}\end{align}
where $n,a_1,a_2,...,a_n \geq 0$ are non-negative integers. If $a_0 \neq 0$
we say that $S_n$ satisfies a {\em full-history recurrence}.
By setting the first $k$ consecutive coefficients to 1 and the remaining
ones to zero, we obtain a sequence of sequences interpolating between
the constant sequence $S_n = 1^n$ for $a_1 = 1, a_k =0$ for $k > 1$, and the
geometric sequence $S_n = 2^n$ obtained if $a_k = 1$ for all $1 \leq k \leq n$.
The Fibonacci numbers are the first non-trivial case, and higher
$m$-nacci numbers are obtained by specifying 
$a_1 = \ldots = a_m = 1$ and $a_k = 0$ for $m < k \leq n$. By allowing other
non-negative integer values of the first $k$ consecutive coefficients
we obtain sequences that
generalize higher-order Fibonacci numbers in a way analogous to the way the
Horadam numbers generalize the classical ones. Hence we call the resulting
sequences the {\em full-history Horadam sequences}. In a similar way we can 
generalize also the sequences defined by recurrences with non-consecutive
non-zero coefficients such as, e.g., the Padovan, the Perrin, the Narayana's
cows sequence, and many others.

Again, we consider a rectangular board of dimension $1\times n$ and
rectangular tiles of dimension $1\times k$ where $k$ can vary from $1$ to $n$,
but now we allow that each type of tile can be colored in a given number of
colors. For the tile of length $k$ we denote the number of allowed colors by
$a_{k}$. The total number of such tilings is denoted by $g_n$. Note that we
do not need to use all tiles from the tile set; if we want to consider
tilings where some size of a tile is omitted, say $j$, we set $a_j=0$.
Again, by Theorem \ref{benq} we have that the number of
all possible tilings is $g_n=S_n$. 

In the following
sections we go beyond the $m$-nacci numbers and establish a master identity
for the sequence $S_n$ whose specializations will yield a number of
known and also of new identities for general sequences defined by homogeneous
linear recurrences with constant coefficients. In particular, we recover
some known, but also establish some new identities for several classical
sequences, including the Fibonacci, the Pell (A000129), the Jacobsthal
(A001045), the Padovan (A000931), the Narayana's cows sequence (A000930),
as well as the tribonacci and the tetranacci
sequences (A000073 and A000078, respectively). Similar work was done
recently by Dresden and Tulskikh \cite{DT1}, but they considered
convolutions involving two sequences. 


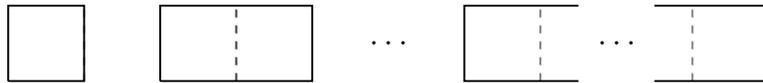
\begin{figure}[h!]
\centering

\begin{tikzpicture}[scale=1]
\coordinate (A) at (0,0);
\coordinate (B) at (1,0);
\coordinate (C) at (2,0);
\coordinate (D) at (3,0);
\coordinate (E) at (4,0);
\coordinate (F) at (5,0);
\coordinate (G) at (6,0);
\coordinate (H) at (7,0);
\coordinate (I) at (8,0);
\coordinate (J) at (9,0);

\coordinate (A1) at (0,1);
\coordinate (B1) at (1,1);
\coordinate (C1) at (2,1);
\coordinate (D1) at (3,1);
\coordinate (E1) at (4,1);
\coordinate (F1) at (5,1);
\coordinate (G1) at (6,1);
\coordinate (H1) at (7,1);
\coordinate (I1) at (8,1);
\coordinate (J1) at (9,1);

\draw [line width=0.25mm] (A)--(B)--(B1)--(A1)--(A); 
\draw [line width=0.25mm] (C)--(E)--(E1)--(C1)--(C); 
\draw [dashed,line width=0.25mm,opacity=0.5] (D)--(D1); 
\node[] () at (5,0.5) {$\cdots$};

\draw [line width=0.25mm] (7.5,0)--(G)--(G1)--(7.5,1); 
\draw [dashed,line width=0.25mm,opacity=0.5] (B)--(B1) (D)--(D1);
\node[] () at (8,0.5) {$\cdots$};
\draw [line width=0.25mm] (8.5,0)--(10,0)--(10,1)--(8.5,1); 
\draw [dashed,line width=0.25mm,opacity=0.5] (J)--(J1) (H)--(H1);
\end{tikzpicture}

\caption{All types of tiles. The leftmost tile is of length $1$ and the rightmost tile of length $m$.}
\label{grid1}
\end{figure}

\section{Combinatorial interpretation of a well-known identity}

As a warm-up, we first provide a combinatorial interpretation of a well-known
identity for the $m$-nacci numbers. The identity itself,
$$F_{n}^{(m)} + F_{n-m-1}^{(m)} = 2 F_{n-1}^{(m)},$$
follows directly from the defining recurrence, but we present a combinatorial 
proof in order to introduce the reader to concepts and to the reasoning
used in obtaining our main results.

It can be easily verified that the $m$-nacci numbers satisfy the identity
$F^{(m)}_n+F^{(m)}_{n-m-1}=2F^{(m)}_{n-1}$. It follows directly from the
recursive relation, 
\begin{align*}
F^{(m)}_n+F^{(m)}_{n-m-1}&=F^{(m)}_{n-1}+F^{(m)}_{n-2}+\cdots+F^{(m)}_{n-m}+F^{(m)}_{n-m-1}\\
&=2F^{(m)}_{n-1}.
\end{align*}

Here we provide a combinatorial interpretation.

\begin{theorem} For $n\geq m+1$, $$F^{(m)}_n+F^{(m)}_{n-m-1}=2F^{(m)}_{n-1}.$$ \label{BenjaminQiunnTheorem}\end{theorem}
\begin{proof} Let $\mathcal{G}_n$ denotes the set of all tilings of a
length-$n$ strip,  and let $\mathcal{T}_n^1, \mathcal{T}_n^2, \cdots
\mathcal{T}_n^m$ denote the tilings ending with a monomer, a dimer,...
or an $m$-mer, respectively. The cardinality of the set $\mathcal{G}_n$
is $g_n=F^{(m)}_{n}$. It is clear that
$\mathcal{G}_n=\mathcal{T}_n^1\cup\mathcal{T}_n^2\cup\cdots\cup\mathcal{T}_n^m$,
where all the sets $\mathcal{T}_n^i,1\geq i\geq m$ are disjoint. To prove
the theorem we have to establish a one-to-one correspondence between the sets
$\mathcal{G}_{n}\cup\mathcal{G}_{n-m-1}$ and $\mathcal{G}_{n-1}\times\left\lbrace0,1\right\rbrace$.

To each tiling from the set $\mathcal{G}_{n-1}$ we add a monomer at the end
to obtain an element of $\mathcal{T}_{n}^1$. Thus, we obtained bijection
between the sets $\mathcal{G}_{n-1}$ and $\mathcal{T}_{n}^1$. In this way,
we have used all the tilings of the set $\mathcal{G}_{n-1}$ once. Now we
take the tilings from the set $\mathcal{T}_{n-1}^m$, and remove the last
$m$-mer to obtain a tiling from a set $\mathcal{G}_{n-m-1}$. This shows
a one-to-one correspondence between those sets. For an arbitrary
$2\leq i \leq m$, we consider the sets $\mathcal{T}_{n}^i$ and
$\mathcal{T}_{n-1}^{i-1}$. Each tiling from $\mathcal{T}_{n}^i$ can be
obtained from a tiling in the set $\mathcal{T}_{n-1}^{i-1}$ as follows:
remove the last $(i-1)$-mer and replace it with an $i$-mer. 

In this way we have used each tiling of length $n-1$ twice and obtained
each tiling of length $n$ and each tiling of length $n-m-1$ exactly once.
A diagram that visualizes the described one-to-one correspondence between the
two sets is shown in Figure \ref{1-1 cor}.  

\begin{figure}[h!]\centering
\begin{tikzpicture}[scale=0.55]

\node[] (p1) at (0,0) {$\mathcal{G}_{n}$};
\node[] (p2) at (5,4) {$\mathcal{T}_{n}^1$};
\node[] (p3) at (5,2) {$\mathcal{T}_{n}^2$};
\node[] (p4) at (5,0) {$\vdots$};
\node[] (p6) at (5,-2) {$\mathcal{T}_{n}^{m}$};
\node[] (p7) at (0,-4) {$\mathcal{G}_{n-m-1}$};
\node[] (p8) at (15,4) {$\mathcal{G}_{n-1}$};
\node[] (p9) at (15,-2) {$\mathcal{G}_{n-1}$};
\node[] (p10) at (10,2) {$\mathcal{T}_{n-1}^1$};
\node[] (p12) at (10,0) {$\vdots$};
\node[] (p13) at (10,-2) {$\mathcal{T}_{n-1}^{m-1}$};
\node[] (p14) at (10,-4) {$\mathcal{T}_{n-1}^{m}$};

\draw [to-to](p2)--(p8);
\draw [to-to](p6)--(p13);
\draw [to-to] (p3)--(p10);
\draw [to-to] (p7)--(p14);
\draw (p1)--(p2) (p1)--(p3) (p1)--(p6) (p9)--(p10) (p9)--(p13) (p9)--(p14);

\end{tikzpicture}  
\caption{One-to-one correspondence between sets $\mathcal{G}_{n}\cup\mathcal{G}_{n-m-1}$ and $\mathcal{G}_{n-1}\times\left\lbrace0,1\right\rbrace$.} \label{1-1 cor}
\end{figure}
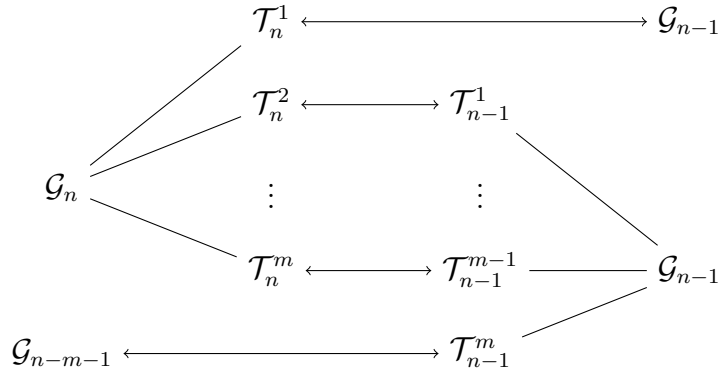 \end{proof}

\section{A master identity for a family of sequences}

Now we turn back to the sequence $S_n$ defined by recursive relation
\ref{recursion S_n}. As mentioned before, it counts the number of
tilings of a board of a length $n$ with tiles of
length $k$ where each of those tiles admits $a_{k}$ colors. 
 Let $\mathcal{T}$ denote the set of all tile lengths considered in a tiling,
i.e., $\mathcal{T}=\left\lbrace i\in [n]: a_{n-i}>0\right\rbrace$. For
example if we want to consider the Fibonacci sequence then
$\mathcal{T}=\left\lbrace 1,2 \right\rbrace$, but if we consider the Pell
numbers the set $\mathcal{T}$ is the same because we care only whether
the number $a_k$ is $0$ or not.

To proceed to our main theorem we need to introduce some new terms. For a
given tiled board, we say that a tiling is \textit{breakable} at position $m$,
where $m\neq 0, m\neq n$, if squares $m$ and $m+1$ (counting from the left) do
not belong to a same tile. That means that a given tiling can be broken into
two smaller properly tiled boards of lengths $m$ and $n-m$, respectively. In
the following example we present a rectangle of length $7$ and the tiles of
length $2$ or $3$. Given tiling can be broken at position $2$ and $5$.

\begin{figure}[h!]
\centering

\begin{tikzpicture}[scale=1]
\coordinate (A) at (0,0);
\coordinate (B) at (1,0);
\coordinate (C) at (2,0);
\coordinate (D) at (3,0);
\coordinate (E) at (4,0);
\coordinate (F) at (5,0);
\coordinate (G) at (6,0);
\coordinate (H) at (7,0);
\coordinate (I) at (8,0);
\coordinate (J) at (9,0);

\coordinate (A1) at (0,1);
\coordinate (B1) at (1,1);
\coordinate (C1) at (2,1);
\coordinate (D1) at (3,1);
\coordinate (E1) at (4,1);
\coordinate (F1) at (5,1);
\coordinate (G1) at (6,1);
\coordinate (H1) at (7,1);
\coordinate (I1) at (8,1);
\coordinate (J1) at (9,1);
 
\draw [line width=0.25mm] (A)--(H)--(H1)--(A1)--(A); 
\draw [dashed,line width=0.25mm,opacity=0.5](B)--(B1)(C)--(C1) (D)--(D1)(E)--(E1)(F)--(F1)(G)--(G1); 
\draw [line width=0.25mm] (A)--(C)--(C1)--(A1)--(A); 
\draw [line width=0.25mm] (C)--(F)--(F1)--(C1)--(C); 
\draw [line width=0.25mm] (F)--(H)--(H1)--(F1)--(F);

\end{tikzpicture}
\caption{Tiling is breakable only at positions $2$ and $5$.}
\label{grid1}
\end{figure}
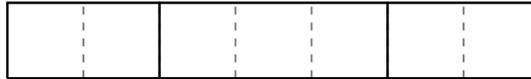 

In our next theorem we obtain an identity that differentiates tilings
based on their breakability.
\begin{theorem} For any integers $m,n\geq 1$ we have the identity \begin{equation}
S_{n+k}=S_{n}S_{k}+\sum\limits_{\substack{i\in \mathcal{T}\\i\neq 1}}a_{i}\sum\limits_{j=1}^{i-1}S_{n-j}S_{k-i+j}. \label{generalizedformula}
\end{equation} \label{generalizedtheorem} \end{theorem}
\begin{proof} There are $S_{n+k}$ tilings of a rectangle of length $n+k$.
We consider an arbitrary tiling from this set. If the tiling is breakable
at position $n$, we divide it into two tiled boards of lengths $n$ and $k$,
respectively. Hence, the total number
of tilings in this case is $S_{n}S_{k}$. If the tiling is not breakable
at position $n$, that means that the tile of length $i$ is blocking it for
some $i\in\mathcal{T}$. Note that there is only one type of tile which can
not block breaking, and that is the tile of length $1$. Hence, $i>1$. Every
such tile occupies $j-1$ positions left of $n$-th square and $i-j$ positions
right of the $n$-th square. The number of such tilings is
$a_{i}S_{n-j}S_{k-i+j}$ since the tile that blocks the breaking can be
colored in $a_i$ ways, the board left of the tile can be colored in
$S_{n-j}$ ways, and the board right in $S_{k-i+j}$ ways. The $n$-th square
can be first in the tile that blocks breaking but can not be last, otherwise
the tiling would be breakable at position $n$. Hence $j$ can vary from $1$
to $i-1$. The total number of tilings in this case is
$a_{i}\sum\limits_{j=1}^{i-1}S_{n-j}S_{k-i+j}$. Since $i$ can be any tile
from a set of tiles $\mathcal{T}$ except the tile of length $1$, we have
the total number
$\sum\limits_{i\in \mathcal{T}, i\neq 1}a_{i}\sum\limits_{j=1}^{i-1}S_{n-j}S_{k-i+j}$. Now we sum up all the cases,
\begin{align*}
S_{n+k}&=S_{n}S_{k}+\sum\limits_{\substack{i\in \mathcal{T}\\i\neq 1}}a_{i}\sum\limits_{j=1}^{i-1}S_{n-j}S_{k-i+j},
\end{align*} which concludes our proof.

\end{proof}  

\section{Some identities obtained from Theorem \ref{generalizedtheorem}}

In this section we demonstrate how Theorem \ref{generalizedtheorem} provides
not only tools for proving some well-known identities, but also for deriving
some new ones. We obtain sequences from the recursive relation \ref{recursion S_n} by defining numbers $a_k$ and by shifting indices, if needed.

As stated in book \textit{Fibonacci and Lucas Numbers with Applications} \cite{Koshy}, next well-known identities that concern Fibonacci numbers were first proven by Lucas in 1876 and Mana in 1969, but here we derive them as a corollary of Theorem \ref{generalizedtheorem}.   
\begin{cor}[Fibonacci numbers] For the Fibonacci numbers we have the following
identities:  $F_{n+k}=F_{n+1}F_{k}+F_{n}F_{k-1}$, $F_{2n+1}=F_{n+1}^2+F_n^2$ and $F_{2n}=(F_{n+1}+F_{n-1})F_{n}$. \label{fibonacci} \end{cor}
\begin{proof}
Since the recursive
relation for Fibonacci numbers is $F_n=F_{n-1}+F_{n-2}$, we set
$a_{n-1}=a_{n-2}=1$ and $a_{k}=0$ otherwise. Thus we have
$\mathcal{T}=\left\lbrace 1,2\right\rbrace$. Now the recursive relation
\ref{recursion S_n} reduces to the recurrence for Fibonacci numbers; the
initial values \ref{initialvalues} need some adjusting which we will do later.
By Theorem \ref{generalizedtheorem} we have:
\begin{align*}
S_{n+k}&=S_{n}S_{k}+\sum\limits_{\substack{i\in \left\lbrace 1,2\right\rbrace\\i\neq 1}}\sum\limits_{j=1}^{i-1}S_{n-j}S_{k-i+j}\\
&=S_{n}S_{k}+S_{n-1}S_{k-1}.
\end{align*}
Since it is usual to set $F_0=0$ for Fibonacci numbers, we shift the index in the array $S_n$
by $1$ to obtain $F_n$. Thus, we have $F_{n+k+1}=F_{n+1}F_{k+1}+F_{n}F_{k}$ and, after replacing $k$
with $k-1$, we obtain a more suitable expression \begin{align}
F_{n+k}=F_{n+1}F_{k}+F_{n}F_{k-1}.
\end{align}
Now by setting $n=k$ we obtain two another famous identities,
\begin{align*}F_{2n+1}&=F_{n+1}^2+F_{n}^2 \end{align*} and \begin{align*}F_{2n}=(F_{n+1}+F_{n-1})F_{n}.\end{align*}
\end{proof}

The Pell numbers are another well-known sequence defined by recursive relation.
Their defining recurrence is $$P_n=2P_{n-1}+P_{n-2},$$ with the initial
values $P_0=0$ and $P_1=1$. 

\begin{cor}[Pell numbers] \label{pell} For Pell numbers we have same identities as in Corollary \ref{fibonacci}. \end{cor}
\begin{proof} Since the Pell numbers satisfy relation
$P_n=2P_{n-1}+P_{n-2}$, we derive new identities for Pell numbers from
recursion \ref{recursion S_n} by setting $a_{1}=2$ and $a_{2}=1$. Since
formula \ref{generalizedformula} does not depend on number $a_1$, all 
identities for Fibonacci numbers are also valid for Pell numbers. Hence, \begin{align*}
P_{n+k}&=P_{n+1}P_{k}+P_{n}P_{k-1},\\ 
P_{2n+1}&=P_{n+1}^2+P_{n}^2,\\ 
P_{2n}&=(P_{n+1}+P_{n-1})P_{n}
\end{align*}  
\end{proof}
As mentioned before, since the identity \ref{generalizedformula} does not
depend on the number $a_1$, all Fibonacci and Pell identities from these
corollaries  can be extended to all sequences $S_n$ that satisfy recursion
$$S_n=aS_{n-1}+S_{n-2},\quad S_{0}=0, S_{1}=1.$$ 

The Jacobsthal numbers $J_n$ are defined by recursion \begin{align*}
J_n=J_{n-1}+2J_{n-2} \label{Jacon recursion} \end{align*} with the initial
values $J_0=0$ and $J_1=1$. The $n$-th element of this sequence can be expressed as the nearest integer of $\frac{2^n}{3}$.
   
\begin{cor}[Jacobsthal numbers] Jacobsthal numbers satisfy the following
identities: $J_{n+k}=J_{n+1}J_{k}+2J_{n}J_{k-1}$, $J_{2n}=(J_{n+1}+2J_{n-1})J_{n}$ and $J_{2n+1}=J_{n+1}^2+2J_{n}^2$ . \label{Jacob}\end{cor}
\begin{proof}

To obtain recursive relation \ref{Jacon recursion}, we set $a_{n-1}=1$, $a_{n-2}=2$ and $a_{k}=0$ otherwise. Again, $\mathcal{T}=\left\lbrace 1,2\right\rbrace$. By Theorem \ref{generalizedtheorem} we have:
\begin{align*}
S_{n+k}&=S_{n}S_{k}+\sum\limits_{\substack{i\in \left\lbrace 1,2\right\rbrace\\i\neq 1}}2\sum\limits_{j=1}^{i-1}S_{n-j}S_{k-i+j}\\
&=S_{n}S_{k}+2S_{n-1}S_{k-1}.
\end{align*}
Because $S_n=J_{n+1}$, we shift indices by $1$ to have $J_{n+k+1}=J_{n+1}J_{k+1}+J_{n}J_{k}$
and, after replacing $k$ with  $k-1$, we obtain \begin{align}
J_{n+k}=J_{n+1}J_{k}+2J_{n}J_{k-1}.
\end{align}
For $n=k$ we have \begin{align*}J_{2n+1}&=J_{n+1}^2+2J_{n}^2 \end{align*} and \begin{align*}J_{2n}=(J_{n+1}+2J_{n-1})J_{n}.\end{align*}
\end{proof}

The result from Corollary \ref{Jacob} can be generalized for all Horadam sequences defined with recursion $$S_n=aS_{n-1}+bS_{n-2},\quad S_{0}=0, S_{1}=1,$$ for arbitrary integers $a$ and $b$. One can easily verify that   \begin{align*}
S_{n+k}=S_{n+1}S_{k}+bS_{n}S_{k-1}.
\end{align*}
  
Now we apply our master identity to sequences with longer defining recurrences.
We start with two sequences defined by recurrences of length three.
The first one, known as the Narayana's cows sequence $N_n$ is defined by the
recurrence
\begin{align*}
N_n=N_{n-1}+N_{n-3}, 
\end{align*} with the initial values $N_0=N_1=N_2=1$. It appears as
A000930 in \cite{OEIS}.

\begin{cor}[Narayana's cows sequence] For the Narayana's cows sequence
we have the
following identity  
$$N_{n+k}=N_{n}N_{k}+N_{n-1}N_{k-2}+N_{n-2}N_{k-1}.$$  \label{narayanacor}\end{cor} \vspace{-1cm}
\begin{proof} We obtain the recursive relation for the Narayana's cows sequence from recursion
\ref{recursion S_n} by setting $a_{n-1}=a_{n-3}=1$ and $a_k=0$ otherwise. 
Since $\mathcal{T}=\left\lbrace 1,3\right\rbrace$ and $S_n=N_n$ we have
\begin{align*}
N_{n+k}&=N_{n}N_{k}+\sum\limits_{i=3}^{3}\sum\limits_{j=1}^{i-1}N_{n-j}N_{k-i+j}\\
&=N_{n}N_{k}+\sum\limits_{j=1}^{2}N_{n-j}N_{k+j-3}\\
&=N_{n}N_{k}+N_{n-1}N_{k-2}+N_{n-2}N_{k-1}
\end{align*}
That concludes
our proof.
\end{proof}

\begin{remark} By setting $n=k$ and shifting indices if needed,
Corollary \ref{narayanacor} allows us to obtain two more identities for
the Narayana's cows sequence: \begin{align*}
N_{2n}&=N_{n}^2+2N_{n-1}N_{n-2}\\
N_{2n+1}&=N_{n-1}^2+N_n\left(N_{n-2}+N_{n+1}\right).
\end{align*}
\end{remark}

Our next identity will we about the Padovan numbers, sequence A000931 in
\cite{OEIS}. Its defining recurrence is given by 
\begin{align*}
P_n=P_{n-2}+P_{n-3},\end{align*}
with the initial values $ P_0=1,P_1=P_2=0$.

\begin{cor}[Padovan sequence] For the Padovan sequence we have following identity \begin{align*}P_{n+k}&=P_{n+1}P_{k+2}+P_{n+2}P_{k+1}+P_{n}P_{k}.
\label{padovanidentity}\end{align*}

 \label{padovancor}\end{cor} \vspace{-1cm}
\begin{proof}
Recursive relation for Padovan numbers is obtained from recursion
\ref{recursion S_n} by setting $a_{n-2}=a_{n-3}=1$ and $a_k=0$ otherwise.
Hence, $\mathcal{T}=\left\lbrace 2,3\right\rbrace$ and from Theorem
\ref{generalizedtheorem} it follows:\begin{align*}
S_{n+k}&=S_{n}S_{k}+\sum\limits_{i=2}^{3}\sum\limits_{j=1}^{i-1}S_{n-j}S_{k-i+j}\\
&=S_{n}S_{k}+S_{n-1}S_{k-1}+S_{n-1}S_{k-2}+S_{n-2}S_{k-1}
\end{align*}
Since $S_n=P_{n+3}$, we have
$$P_{n+k+3}=P_{n+3}P_{k+3}+P_{n+2}P_{k+2}+P_{n+2}P_{k+1}+P_{n+1}P_{k+2},$$
and by replacing $n$ and $k$ with $n-2$ and $k-1$ we finally arrive to \begin{align*}
P_{n+k}&=P_{n+1}P_{k+2}+\left(P_n+P_{n-1}\right)P_{k+1}+P_{n}P_{k}\\
&=P_{n+1}P_{k+2}+P_{n+2}P_{k+1}+P_{n}P_{k}.
\end{align*}
\end{proof}

\begin{remark} Similar as before, Corollary \ref{padovancor} allows us to write two more identities: \begin{align*}
P_{2n}&=P_{n}^2+2P_{n+1}P_{n+2}\\
P_{2n+1}&=P_{n+2}^2+P_{n+1}P_{n+3}+P_nP_{n+1}.
\end{align*}
\end{remark}

Although the list of possibilities is vast, we will conclude this list of
examples with finite history recursion with the tribonacci and tetranacci
numbers, denoted by $T_n$ and $Q_n$, respectively. Their respective defining 
recurrence relations are \begin{align*}
T_n=T_{n-1}+T_{n-2}+T_{n-3}
\end{align*} and \begin{align*}
Q_n=Q_{n-1}+Q_{n-2}+Q_{n-3}+Q_{n-4}
\end{align*}
with initial values $T_0=T_1=0$, $T_2=1$ and $Q_0=Q_1=Q_2=0$, $Q_3=1$.

The identity for tribonacci numbers was first established by other means
in our recent paper \cite{DP}, but here we show how can it be derived from
our main result. 

\begin{cor}[Tribonacci numbers] Tribonacci numbers satisfy the identity
$$T_{n+k}=T_{n+1}T_{k+1}+T_{n}T_{k}+T_{n}T_{k-1}+T_{n-1}T_{k}.$$ \end{cor}
\begin{proof}
We derive recursive relation for tribonacci numbers from recursion
\ref{recursion S_n} by setting $a_{n-1}=a_{n-2}=a_{n-3}=1$ and $a_k = 0$
otherwise. This yields $\mathcal{T}=\left\lbrace 1,2,3\right\rbrace$ and
we have \begin{align*}
S_{n+k}&=S_{n}S_{k}+\sum\limits_{i=2}^{3}\sum\limits_{j=1}^{i-1}S_{n-j}S_{k-i+j}\\
&=S_{n}S_{k}+\sum\limits_{j=1}^{1}S_{n-j}S_{k+j-2}+\sum\limits_{j=1}^{2}S_{n-j}S_{k+j-3}\\
&=S_{n}S_{k}+S_{n-1}S_{k-1}+S_{n-1}S_{k-2}+S_{n-2}S_{k-1}.
\end{align*}
Because $S_n=T_{n+2}$, we shift indices by $2$ and after replacing $n$ and $k$ with $n-1$ and $k-1$, respectively,
we obtain a more elegant expression
$$T_{n+k}=T_{n+1}T_{k+1}+T_{n}T_{k}+T_{n}T_{k-1}+T_{n-1}T_{k}.$$ 
\end{proof}

Dresden and Jin already proved a number of tetranacci identities in their recent paper \cite{Dresden}. Here we provide yet another tetranacci identity.

\begin{cor}[Tetranacci numbers] The Tetranacci numbers satisfy
$$Q_{n+k}=Q_{n+2}Q_{k+1}+Q_{n+1}\left(Q_{k+2}-Q_{k+1}\right)+Q_{n}\left(Q_{k-1}+Q_k\right)+Q_{n-1}Q_{k}.$$ \end{cor}
\begin{proof}
We derive recursive relation for tetranacci numbers from recursion
\ref{recursion S_n} by setting $a_{n-1}=a_{n-2}=a_{n-3}=a_{n-4}=1$ as the
only non-zero coefficients. This yields
$\mathcal{T}=\left\lbrace 1,2,3,4\right\rbrace$ and we have
\begin{align*}
S_{n+k}&=S_{n}S_{k}+\sum\limits_{i=2}^{4}\sum\limits_{j=1}^{i-1}S_{n-j}S_{k-i+j}\\
&=S_{n}S_{k}+\sum\limits_{j=1}^{1}S_{n-j}S_{k-2+j}+\sum\limits_{j=1}^{2}S_{n-j}S_{k-3+j}+\sum\limits_{j=1}^{3}S_{n-j}S_{k-4+j}\\
&=S_{n}S_{k}+S_{n-1}S_{k-1}+\sum\limits_{j=1}^{2}S_{n-j}S_{k-3+j}+\sum\limits_{j=1}^{3}S_{n-j}S_{k-4+j}\\
&=S_{n}S_{k}+S_{n-1}\left(S_{k-1}+S_{k-2}+S_{k-3}\right)+S_{n-2}S_{k-1}+S_{n-2}S_{k-2}+S_{n-3}S_{k-1}\\
&=S_{n}S_{k}+S_{n-1}\left(S_{k+1}-S_{k}\right)+S_{n-2}S_{k-1}+S_{n-2}S_{k-2}+S_{n-3}S_{k-1}.
\end{align*}
By shifting indices by $3$ and replacing $n$ with $n-1$ and $k$ with $k-2$
we arrive at
$$Q_{n+k}=Q_{n+2}Q_{k+1}+Q_{n+1}\left(Q_{k+2}-Q_{k+1}\right)+Q_{n}\left(Q_{k-1}+Q_k\right)+Q_{n-1}Q_{k}.$$ 
\end{proof}

\begin{remark} For tribonacci and tetranacci numbers we have following identities: \begin{align*}
T_{2n}&=T_{n+1}^2+T_n^2+2T_{n}T_{n-1},\\
T_{2n+1}&=T_{n}^2-2T_{n+1}^2+2T_{n+1}T_{n+2},\\
Q_{2n}&=2Q_{n+2}Q_{n+1}-Q_{n+1}^2+Q_n^2+2Q_nQ_{n-1},\\
Q_{2n+1}&=Q_{n+2}^2-Q_{n+1}^2+Q_{n}^2+2Q_{n+1}\left(Q_{n+3}-Q_{n+2}\right).
\end{align*}
\end{remark}

Our last example involves a full-history recursion. The result can be also
obtained from Corollary \ref{fibonacci}, but here we want to present an
application of Theorem \ref{generalizedtheorem} in case where
$\mathcal{T}=\mathbb{N}$. We tile a rectangular strip of length $n$ with
tiles of all possible sizes that vary from $1$ to $n$ where each tile of
length $k$ admits $k$ colors. So, let $a_k=k$ for $k\geq 0$. 
We obtain recursive relation
$$S_n=1\cdot S_{n-1}+2\cdot S_{n-2}+\cdots +n\cdot S_0=\sum\limits_{i=0}^{n-1}(n-i) S_{i}.$$
Thus we have $\mathcal{T}=\mathbb{N}$. It is not hard to see that $S_1=1$,
$S_2=3$, $S_3=8$ and in general $S_n=F_{2n}$, where $F_n$, as before,
denotes $n$-th Fibonacci number. By Theorem \ref{generalizedtheorem} we have

$$S_{n+k}=S_{n}S_{k}+\sum\limits_{i=2}^{\infty}i\sum\limits_{j=1}^{i-1}S_{n-j}S_{k-i+j},$$
and since $S_k=0$ for $k<0$, we can restrict our range of summation. Hence,

\begin{align*}
S_{n+k}&=S_{n}S_{k}+\sum\limits_{i=2}^{n+k}i\sum\limits_{j=1}^{i-1}S_{n-j}S_{k-i+j}\\
&=S_{n}S_{k}+\sum\limits_{j=1}^{n+k-1}S_{n-j}\sum\limits_{i=j+1}^{n+k}iS_{k-i+j}\\
&=S_{n}S_{k}+\sum\limits_{j=1}^{n}S_{n-j}\sum\limits_{i=j+1}^{j+k}iS_{k-i+j}.\\
\end{align*}
Since \begin{align*}
\sum\limits_{j=1}^{n}S_{n-j}\sum\limits_{i=j+1}^{j+k}iS_{k-i+j}&=\sum\limits_{j=1}^{n}S_{n-j}\left((j+1)S_{k-1}+\cdots+(j+k)S_{0}\right)\\
&=\sum\limits_{j=1}^{n}S_{n-j}\left(j\left(S_{k-1}+\cdots+S_{0}\right)+1\cdot S_{k-1}+\cdots+k\cdot S_{0}\right)\\
&=\sum\limits_{j=1}^{n}S_{n-j}\left(j\sum\limits_{j=0}^{k-1}S_{j}+S_{k}\right)\\
&=S_{n}\sum\limits_{j=0}^{k-1}S_{j}+S_{k}\sum\limits_{j=0}^{n-1}S_{j}\\
&=S_{n}\sum\limits_{j=1}^{k-1}S_{j}+S_{k}\sum\limits_{j=1}^{n-1}S_{j}+S_{n}+S_{k},\\
\end{align*} 
we finally have  \begin{align*}S_{n+k}=S_nS_k+S_{n}\sum\limits_{j=1}^{k-1}S_{j}+S_{k}\sum\limits_{j=1}^{n-1}S_{j}+S_{n}+S_{k}.
\end{align*} 
All indices above are greater than $0$, so we can replace $S_{n}$ with $F_{2n}$
and, after using identity $\sum\limits_{j=1}^n F_{2j}=F_{2n+1}-1$,  we obtain
\begin{align*}
F_{2n+2k}&=F_{2n}F_{2k}+F_{2n}\sum\limits_{j=1}^{k-1}F_{2j}+F_{2k}\sum\limits_{j=1}^{n-1}F_{2j}+F_{2n}+F_{2k}\\
&=F_{2n}F_{2k}+F_{2n}\left(F_{2k-1}-1\right)+F_{2k}\left(F_{2n-1}-1\right)+F_{2n}+F_{2k}\\
&=F_{2n}F_{2k}+F_{2n}F_{2k-1}+F_{2k}F_{2n-1}\\
&=F_{2n}F_{2k+1}+F_{2k}F_{2n-1}.
\end{align*} 
Finally, and after setting $n=k$,
 \begin{align*}
F_{4n}&=F_{2n}\left(F_{2n+1}+F_{2n-1}\right)\\
&=F_{n}(F_{n+1}+F_{n-1})\left(F_{n+1}^2+2F_n^2+F_{n-1}^2\right).
\end{align*}

\section{Concluding remarks}

In this paper we have considered sequences defined by full-history homogeneous
linear recurrences with non-negative constant integer coefficients. Such 
sequences encompass in a unifying frame several known generalizations of
Fibonacci numbers, including, among others, the Horadam sequences and the
$m$-nacci numbers. A master identity is established for the generalized
sequences and a number of identities for many combinatorially interesting
sequences are obtained as its corollaries. We mention that our approach
could be successfully employed to derive further identities by considering
several breaking points, leading thus to results for sequences indexed by
sums of three and more indices.

\nocite{*}
\bibliographystyle{amsplain}
\bibliography{}

\end{document}